\documentclass[12pt,leqno]{amsart}
\usepackage{amsmath,amsfonts,latexsym,graphicx,amssymb,url, color}
\usepackage{hyperref}
\usepackage{txfonts} 
\newcommand{\average}{{\mathchoice {\kern1ex\vcenter{\hrule
height.4pt width 6pt depth0pt} \kern-9.7pt}
{\kern1ex\vcenter{\hrule height.4pt width 4.3pt depth0pt}
\kern-7pt} {} {} }}

\newcommand{\abs}[1]{\left\vert#1\right\vert}

\newcommand{\R}{{\mathbb R}} 

\newcommand{\e}{\varepsilon}

\newcommand{\p}{\partial}

\newcommand{\norm}[1]{\left\Vert#1\right\Vert}

\newenvironment{myindentpar}[1]%
{\begin{list}{}%
         {\setlength{\leftmargin}{#1}}%
         \item[]%
}
{\end{list}}
\theoremstyle{plain}
\newtheorem{theorem}{Theorem}[section]

\newtheorem{remark}[theorem]{Remark}

\title[Boundary H\"older gradient estimates for linearized Monge-Amp\`ere]{
On boundary H\"older gradient estimates for solutions to the linearized Monge-Amp\`ere equations
}
\author{Nam Q. Le}
\address{Department of Mathematics, Columbia University, New York, NY 10027, USA}
\email{namle@math.columbia.edu}
\curraddr{Institute of Mathematics, Vietnam Academy of Science and Technology, 18 Hoang Quoc Viet, 10307 Hanoi, Vietnam}
\author{Ovidiu Savin}
\address{Department of Mathematics, Columbia University, New York, NY 10027, USA}
\email{savin@math.columbia.edu}

\begin{document}
 \maketitle

\begin{abstract}
In this paper, we establish boundary H\"older gradient estimates for solutions to the linearized Monge-Amp\`ere equations with $L^{p}$ ($n<p\leq\infty$)
right hand side and $C^{1,\gamma}$ boundary values
under natural assumptions on the domain, boundary data and the
Monge-Amp\`ere measure. These estimates extend our previous boundary regularity results for solutions to the linearized Monge-Amp\`ere equations
with bounded 
right hand side and $C^{1, 1}$ boundary data. 
\end{abstract}
\noindent

\section{Statement of the main results}
In this paper, we establish boundary H\"older gradient estimates for solutions to the linearized Monge-Amp\`ere equations 
with $L^{p}$ ($n<p\leq\infty$)
right hand side and $C^{1,\gamma}$ boundary values under 
natural assumptions on the domain, boundary data and the
Monge-Amp\`ere measure. 
Before stating these estimates, we introduce 
the following assumptions on the domain $\Omega$ and function
$\phi$.

Let $\Omega\subset \R^{n}$ be a bounded convex set with
\begin{equation}\label{om_ass}
B_\rho(\rho e_n) \subset \, \Omega \, \subset \{x_n \geq 0\} \cap B_{\frac 1\rho},
\end{equation}
for some small $\rho>0$. Assume that 
\begin{equation}
\Omega~ \text{contains an interior ball of radius $\rho$ tangent to}~ \p 
\Omega~ \text{at each point on} ~\p \Omega\cap\ B_\rho.
\label{tang-int}
\end{equation}
Let $\phi : \overline \Omega \rightarrow \R$, $\phi \in C^{0,1}(\overline 
\Omega) 
\cap 
C^2(\Omega)$  be a convex function satisfying
\begin{equation}\label{eq_u}
 0 <\lambda \leq \det D^2\phi \leq \Lambda \quad \text{in $\Omega$}.
\end{equation}
Throughout, we denote by $\Phi = (\Phi^{ij})$ the matrix of cofactors of the 
Hessian matrix $D^{2}\phi$, 
i.e., $$\Phi = (\det D^{2} \phi) (D^{2} \phi)^{-1}.$$
We assume that on $\p \Omega\cap B_\rho$, 
$\phi$ separates quadratically from its tangent planes on $\p \Omega$. 
Precisely we assume that if $x_0 \in 
\p \Omega \cap B_\rho$ then
\begin{equation}
 \rho\abs{x-x_{0}}^2 \leq \phi(x)- \phi(x_{0})-\nabla \phi(x_{0}) (x- x_{0}) \leq 
\rho^{-1}\abs{x-x_{0}}^2,
\label{eq_u1}
\end{equation}
for all $x \in \p\Omega.$ 

Let $S_{\phi}(x_0, h)$ be the section of $\phi$ centered at $x_0\in \overline{\Omega}$ and of height $h$:
$$S_{\phi}(x_0, h):= \{x\in \overline{\Omega}: \phi(x)<\phi(x_0)+\nabla\phi(x_0)(x-x_0)+ h\}.$$
When $x_0$ is the origin, we denote for simplicity $S_h:= S_{\phi}(0, h).$

Now, we can state our boundary H\"older gradient estimates for solutions to the linearized Monge-Amp\`ere equations with $L^{p}$ right hand side and 
$C^{1,\gamma}$ boundary data.
\begin{theorem}
\label{h-bdr-gradient}
Assume $\phi$ and $\Omega$ satisfy the assumptions 
\eqref{om_ass}-\eqref{eq_u1} 
above. Let $u: B_{\rho}\cap 
\overline{\Omega}\rightarrow \R$ be a continuous solution to 
\begin{equation*}
 \left\{
 \begin{alignedat}{2}
   \Phi^{ij}u_{ij} ~& = f ~&&\text{in} ~ B_{\rho}\cap \Omega, \\\
u &= \varphi~&&\text{on}~\p \Omega \cap B_{\rho},
 \end{alignedat} 
  \right.
\end{equation*} 
where $f\in L^{p}(B_{\rho}\cap\Omega)$ for some $p>n$ and $\varphi \in C^{1,\gamma}(B_{\rho}\cap\p\Omega)$.
Then, there exist $\alpha\in (0, 1)$ and $\theta_0$ small depending only on $n, p, \rho, \lambda, \Lambda, \gamma$ such that for all $\theta\leq \theta_0$ we have 
$$\|u- u(0)-\nabla u(0)x\|_{L^{\infty}(S_{\theta})}\leq C\left(\|u\|_{L^{\infty}(B_{\rho}\cap\Omega)} + \|f\|_{L^{p}(B_{\rho}\cap\Omega)} + 
 \|\varphi\|_{C^{1,\gamma}(B_{\rho}\cap\p\Omega)}\right) (\theta^{1/2})^{1+\alpha}$$ 
where $C$ only on $n, p, \rho, \lambda, \Lambda, \gamma$. We can take $\alpha:= \min\{1-\frac{n}{p}, \gamma\}$ provided that
$\alpha<\alpha_0$
where $\alpha_0$ is the exponent in our previous boundary H\"older gradient estimates (see Theorem \ref{LS-gradient}).
\end{theorem}
Theorem \ref{h-bdr-gradient} extends our previous boundary H\"older gradient estimates 
for solutions to the linearized Monge-Amp\`ere equations
with bounded 
right hand side and $C^{1, 1}$ boundary data \cite[Theorem 2. 1]{LS1}. This is an affine invariant analogue of the boundary H\"older gradient estimates of  Ural'tseva  \cite{U1} (see also \cite{U2} for a survey) for uniformly elliptic equation with $L^{p}$ right hand side.
\begin{remark}
 By the Localization Theorem \cite{S1, S2}, we have
 $$B_{c\theta^{1/2}/\abs{log \theta}}\cap \overline{\Omega}\subset S_{\theta}\subset B_{C\theta^{1/2}\abs{log \theta}}\cap\overline{\Omega}.$$
 Therefore, Theorem \ref{h-bdr-gradient} easily implies that $\nabla u$ is $C^{0, \alpha^{'}}$ on $B_{\rho/2}\cap\p\Omega$ for all $\alpha^{'}<\alpha.$
\end{remark}

As a consequence of Theorem \ref{h-bdr-gradient}, we obtain global $C^{1,\alpha}$ estimates for solutions to the linearized Monge-Amp\`ere equations 
with $L^{p}$ ($n<p\leq\infty$)
right hand side and $C^{1,\gamma}$ boundary values under natural assumptions on the domain, boundary data and continuity of the
Monge-Amp\`ere measure.
\begin{theorem}
Assume that 
$\Omega \subset B_{1/\rho}$ contains an interior ball of radius $\rho$ 
tangent to $\p 
\Omega$ at each point on $\p \Omega.$ Let $\phi : \overline \Omega \rightarrow \R$, $\phi \in C^{0,1}(\overline 
\Omega) 
\cap 
C^2(\Omega)$  be a convex function satisfying
$$ \det D^2 \phi =g \quad \quad \mbox{with} \quad \lambda \le g \le \Lambda,\quad g\in C(\overline{\Omega}).$$ Assume further that 
on $\p \Omega$, $\phi$ 
separates quadratically from its 
tangent planes,  namely
\begin{equation*}
 \rho\abs{x-x_{0}}^2 \leq \phi(x)- \phi(x_{0})-\nabla \phi(x_{0}) (x- x_{0})
 \leq \rho^{-1}\abs{x-x_{0}}^2, ~\forall x, x_{0}\in\p\Omega.
\end{equation*}
Let $u: \overline{\Omega}\rightarrow \R$ be a continuous function that 
solves the linearized Monge-Amp\`ere equation 
\begin{equation*}
 \left\{
 \begin{alignedat}{2}
   \Phi^{ij}u_{ij} ~& = f ~&&\text{in} ~ \Omega, \\\
u &= \varphi ~&&\text{on}~\p \Omega,
 \end{alignedat} 
  \right.
\end{equation*}
where $\varphi$ is a $C^{1,\gamma}$ function defined on $\p\Omega$ $(0<\gamma\leq 1)$ and $f\in L^{p}(\Omega)$ with $p>n$. 
Then
\begin{equation*}
 \|u\|_{C^{1, \beta} (\overline \Omega )} \leq K( \|\varphi\|_{C^{1,\gamma}(\p \Omega)}+ \|f\|_{L^p(\Omega)}), 
\end{equation*}
where $\beta\in (0,1)$ and $K$ are constants depending on 
$n, \rho, \gamma, \lambda, \Lambda, p$ and the modulus of continuity of $g$.
 \label{global-reg}
\end{theorem}

Theorem \ref{global-reg} extends our previous global $C^{1,\alpha}$ estimates 
for solutions to the linearized Monge-Amp\`ere equations
with bounded 
right hand side and $C^{1, 1}$ boundary data \cite[Theorem 2. 5 and Remark 7.1]{LS1}. It is also the global counterpart of Guti\'errez-Nguyen's interior $C^{1,\alpha}$
estimates for the linearized Monge-Amp\`ere equations. If we assume $\varphi$ to be more regular, say $\varphi \in W^{2, q}(\Omega)$ where $q>p$,
then Theorem \ref{global-reg} is a consequence of the global $W^{2, p}$ estimates for solutions to the linearized Monge-Amp\`ere 
equations \cite[Theorem 1. 2]{LN}. In this case, the proof in \cite{LN} is quite involved. Our proof of Theorem \ref{global-reg} here is much simpler.
\begin{remark}
The estimates in Theorem \ref{global-reg} can be improved to  
\begin{equation}
 \|u\|_{C^{1, \beta} (\overline \Omega )} \leq K( \|\varphi\|_{C^{1,\gamma}(\p \Omega)}+ \|f/tr~ \Phi\|_{L^p(\Omega)}).
\label{improved-est}
 \end{equation}
This follow easily from the estimates in Theorem \ref{global-reg} and the global $W^{2,p}$ estimates for solutions to the standard Monge-Amp\`ere equations with continuous right hand side
 \cite{S3}. Indeed, since $$tr ~\Phi\geq n(\det \Phi)^{\frac{1}{n}}\geq n \lambda^{\frac{n-1}{n}},$$  we also have 
 $f/tr ~\Phi\in L^{p} (\Omega)$. Fix $q\in (n, p)$, then by \cite{S3}, $tr~ \Phi\in L^{\frac{pq}{p-q}}(\Omega)$. 
 Now apply the estimates in Theorem \ref{global-reg} to $f\in L^{q}(\Omega)$ and then use H\"older inequality to 
 $f= (f/tr~ \Phi) (tr~ \Phi)$ to obtain (\ref{improved-est}).
\end{remark}
\begin{remark}
The linearized Monge-Amp\`ere operator $L_{\phi}:= \Phi^{ij}\p_{ij}$ with $\phi$ satisfying the conditions of Theorem \ref{global-reg} is in general degenerate. Here is an explicit example in two dimensions, taken from \cite{Wang1}, showing that $L_{\phi}$ is not uniformly elliptic in $\overline{\Omega}.$ Consider
$$\phi (x, y)=\frac{x^2}{log |log (x^2 + y^2)|} + y^2 log|log (x^2 + y^2)| $$
in a small ball $\Omega= B_{\rho}(0)\subset R^2$ around the origin. Then $\phi \in C^{0,1}(\overline 
\Omega) 
\cap 
C^2(\Omega)$ is strictly convex with
$$\det D^2 \phi(x, y) = 4 + O (\frac{log|log(x^2 + y^2)|}{log(x^2 + y^2)})\in C(\overline{\Omega})$$
and $\phi$ has smooth boundary data on $\p\Omega$. The quadratic separation of $\phi$ from its tangent planes on $\p\Omega$ can be readily checked (see also \cite[Proposition 3.2]{S2}). However $\phi\not\in W^{2,\infty}(\Omega).$
\end{remark}
\begin{remark}
For the global $C^{1,\alpha}$ estimates in Theorem \ref{global-reg}, the condition $p>n$ is sharp, since even in the uniformly elliptic case (for example, when $\phi(x)= \frac{1}{2}|x|^2$, $L_{\phi}$ is the Laplacian), the global $C^{1,\alpha}$ estimates fail when $p=n$.
\end{remark}

We prove Theorem \ref{h-bdr-gradient} using the perturbation arguments in the spirit of Caffarelli \cite{C, CC} (see also Wang \cite{Wang})
in combination with our previous boundary H\"older gradient estimates for the case of bounded right hand side $f$ and $C^{1,1}$ boundary data \cite{LS1}.

The next section will provide the proof of Theorem \ref{h-bdr-gradient}. The proof of Theorem \ref{global-reg} will be given in the final section, Section 
\ref{proof-sec}.
\section{Boundary H\"older gradient estimates}

In this section, we prove Theorem \ref{h-bdr-gradient}. {\it We will use the letters $c, C$ to denote generic constants depending only on 
the structural constants $n, p, \rho, \gamma, \lambda, \Lambda$ that may 
change from line to line.}

Assume $\phi$ and $\Omega$ satisfy the assumptions 
\eqref{om_ass}-\eqref{eq_u1}. We can also assume that $\phi(0)=0$ and $\nabla \phi(0)=0.$ 
By the Localization Theorem for solutions to the Monge-Amp\`ere equations proved in \cite{S1, S2}, there exists a small constant $k$ 
depending only on $n, \rho, \lambda, \Lambda$ such that if $h\leq k$ then 
\begin{equation}kE_h\cap \overline{\Omega}\subset S_{\phi}(0, h)\subset k^{-1} E_h\cap \overline{\Omega}\
 \label{loc-k}
\end{equation}
where
$$E_h:= h^{1/2}A_h^{-1} B_1$$
with $A_h$ being a linear transformation (sliding along the $x_n=0$ plane) 
\begin{equation}A_h(x) = x- \tau_h x_n,~ \tau_h\cdot e_n =0, ~\det A_h =1\
 \label{Amap}
\end{equation}
and
$$|\tau_h|\leq k^{-1}\abs{log h}.$$
We define the following rescaling of $\phi$
\begin{equation}\phi_h(x):= \frac{\phi(h^{1/2} A^{-1}_h x)}{h}
 \label{phi-h}
\end{equation}
in \begin{equation}\Omega_h:= h^{-1/2}A_h\Omega.
    \label{omega-h}
   \end{equation}
Then
$$\lambda \leq \det D^2 \phi_h(x)= \det D^2 \phi(h^{1/2}A_{h}^{-1}x)\leq \Lambda$$
and
$$B_{k}\cap \overline{\Omega_h}\subset S_{\phi_h}(0, 1)= h^{-1/2} A_{h}S_h\subset B_{k^{-1}}\cap \overline{\Omega_h}.$$
Lemma 4. 2 in \cite{LS1} implies that if $h, r\leq c$ small then $\phi_h$ satisfies in $S_{\phi_h}(0, 1)$ the hypotheses of the Localization Theorem
\cite{S1, S2}
at all $x_0\in S_{\phi_h}(0, r)\cap\p S_{\phi_h}(0, 1).$ In particular, there exists $\tilde\rho$ small, depending only on $n, \rho, \lambda, \Lambda$ such that
if $x_0\in S_{\phi_h}(0, r)\cap\p S_{\phi_h}(0, 1)$ then
\begin{equation}
 \tilde\rho\abs{x-x_{0}}^2 \leq \phi_h(x)- \phi_h(x_{0})-\nabla \phi_h(x_{0}) (x- x_{0}) \leq 
\tilde\rho^{-1}\abs{x-x_{0}}^2,
\label{loc-h}
\end{equation}
for all $x \in \p S_{\phi_h}(0, 1).$ We fix $r$ in what follows.

Our previous boundary H\"older gradient estimates \cite{LS1} for solutions to the linearized Monge-Amp\`ere with 
bounded 
right hand side and $C^{1, 1}$ boundary data hold in $S_{\phi_h}(0, r)$. They will play a crucial role in the perturbation arguments
and we now recall them here.
\begin{theorem}(\cite[Theorem 2.1 and Proposition 6.1]{LS1})
Assume $\phi$ and $\Omega$ satisfy the assumptions 
\eqref{om_ass}-\eqref{eq_u1} 
above. Let $u: S_r\cap 
\overline{\Omega}\rightarrow \R$ be a continuous solution to 
\begin{equation*}
 \left\{
 \begin{alignedat}{2}
   \Phi^{ij}u_{ij} ~& = f ~&&\text{in} ~ S_r\cap \Omega, \\\
u &= 0~&&\text{on}~\p \Omega \cap S_r,
 \end{alignedat} 
  \right.
\end{equation*} 
where $f\in L^{\infty}(S_r\cap\Omega)$.
Then 
$$|\p_n u(0)|\leq C_0 \left(\|u\|_{L^{\infty}(S_r\cap\Omega)} + \|f\|_{L^{\infty}(S_r\cap\Omega)}\right)$$
and for $s\leq r/2$
$$\max_{S_r}|u-\p_n u(0)x_n|\leq C_0 (s^{1/2})^{1+\alpha_0}\left(\|u\|_{L^{\infty}(S_r\cap\Omega)} + \|f\|_{L^{\infty}(S_r\cap\Omega)} \right)$$
where $\alpha_0\in (0, 1)$ and $C_0$ are constants depending only on $n, \rho,
 \lambda, \Lambda $.

\label{LS-gradient}
\end{theorem}

Now, we are ready to give the proof of Theorem \ref{h-bdr-gradient}.
\begin{proof}[Proof of Theorem \ref{h-bdr-gradient}]
Since $u|_{\p\Omega \cap B_{\rho}}$ is $C^{1,\gamma}$, by subtracting a suitable linear function we can assume that on $\p\Omega\cap B_{\rho}$, $u$
satisfies
$$\abs{u(x)}\leq M |x^{'}|^{1+\gamma}.$$
Let $$\alpha:=\min\{\gamma, 1-\frac{n}{p}\}$$ if $\alpha<\alpha_0$; otherwise let $\alpha\in (0, \alpha_0)$ where $\alpha_0$ is in Theorem \ref{LS-gradient}.
The only place where we need $\alpha<\alpha_0$ is (\ref{alpha0}).

By dividing our equation by a suitable constant we may assume that for some $\theta$ to be chosen later
$$\|u\|_{L^{\infty}(B_{\rho}\cap\Omega)} + \|f\|_{L^{p}(B_{\rho}\cap\Omega)} + M\leq (\theta^{1/2})^{1+\alpha}=: \delta.$$
{\bf Claim.} There exists $0<\theta_0<r/4$ small depending only on $n, \rho, \lambda, \Lambda, \gamma, p$, and a sequence of linear functions
$$l_m(x):= b_m x_n$$
with where $b_0= b_1 =0$ such that for all $\theta\leq\theta_0$ and for all $m\geq 1$, we have
\begin{myindentpar}{1cm}
 (i) $$\|u-l_m\|_{L^{\infty}(S_{\theta^m})}\leq (\theta^{m/2})^{1+\alpha},$$
 and\\
 (ii) $$|b_m-b_{m-1}|\leq C_0 (\theta^{\frac{m-1}{2}})^{\alpha}.$$
\end{myindentpar}
Our theorem follows from the claim. Indeed, (ii) implies that $\{l_m\}$ converges uniformly in $S_{\theta}$ to a linear function $l(x)= bx_n$ with $b$ 
universally bounded since
\begin{equation*}
\abs{b}\leq  \sum_{m=1}^{\infty} \abs{b_m-b_{m-1}}\leq \sum_{m=1}^{\infty} C_0 (\theta^{\theta/2})^{m-1}= \frac{C_0}{1-\theta^{\alpha/2}}\leq 2C_0.
\end{equation*}
Furthermore, by (\ref{loc-k}) and (\ref{Amap}), we have $|x_n|\leq k^{-1}\theta^{m/2}$ for $x\in S_{\theta^m}$. Therefore, 
for any $m\geq 1$, 
 \begin{eqnarray*}
  \|u-l\|_{L^{\infty}(S_{\theta^m})} &\leq &  \|u-l_m\|_{L^{\infty}(S_{\theta^m})} + \sum_{j=m+1}^{\infty}  \|l_j-l_{j-1}\|_{L^{\infty}(S_{\theta^m})}
  \\&\leq& (\theta^{m/2})^{1+\alpha} + \sum_{j=m+1}^{\infty}  C_0 (\theta^{\frac{j-1}{2}})^{\alpha} (k^{-1}\theta^{m/2})\\
  &\leq& C(\theta^{m/2})^{1+\alpha}.
 \end{eqnarray*}
We now prove the claim by induction. Clearly (i) and (ii) hold for $m=1$. Suppose (i) and (ii) hold up to $m\geq 1$. We prove them for $m+1$. 
Let $h= \theta^m$. We define 
the rescaled domain $\Omega_h$ and function $\phi_h$ as in (\ref{omega-h}) and (\ref{phi-h}).  We also define
for $x\in \Omega_h$
$$v(x):= \frac{(u-l_m)(h^{1/2} A^{-1}_h x)}{h^{\frac{1+\alpha}{2}}}, ~f_h(x):
= h^{\frac{1-\alpha}{2}}f(h^{1/2} A^{-1}_h x).$$
Then
$$\|v\|_{L^{\infty}(S_{\phi_h}(0,1))}= \frac{1}{h^{\frac{1+\alpha}{2}}}\|u-l_m\|_{L^{\infty}(S_h)}\leq 1$$
and
$$\Phi_h^{ij}v_{ij}=f_h~\text{in}~S_{\phi_h}(0,1)$$
with
$$\|f_h\|_{L^{p}(S_{\phi_h}(0,1))} = (h^{1/2})^{1-\alpha-n/p} \|f\|_{L^{p}(S_h)}\leq \delta.$$
Let $w$ be the solution to
\begin{equation*}
 \left\{
 \begin{alignedat}{2}
   \Phi_h^{ij}w_{ij} ~& = 0 ~&&\text{in} ~ S_{\phi_h}(0, 2\theta), \\\
w &= \varphi_h~&&\text{on}~\p S_{\phi_h}(0, 2\theta),
 \end{alignedat} 
  \right.
\end{equation*} 
where 
\begin{equation*}
 \varphi_h = \left\{\begin{alignedat}{1}
                  0 ~&~ \text{on} ~\p S_{\phi_h}(0, 2\theta)\cap\p\Omega_h \\
                  v~& ~\text{on}~ \p S_{\phi_h}(0, 2\theta) \cap \Omega_h.
                 \end{alignedat}
\right.
\end{equation*}
By the maximum principle, we have
$$\|w\|_{L^{\infty}(S_{\phi_h}(0, 2\theta))}\leq \|v\|_{L^{\infty}(S_{\phi_h}(0, 2\theta))}\leq 1.$$
Let 
$$\bar{l}(x):= \bar{b}x_n;~\bar{b}:=\partial_n w(0).$$
Then the boundary H\"older gradient estimates in Theorem \ref{LS-gradient} give
\begin{equation}\abs{\bar{b}}\leq C_0 \|w\|_{L^{\infty}(S_{\phi_h}(0, 2\theta))}\leq C_0
 \label{b-bar}
\end{equation}
and
\begin{eqnarray}\|w-\bar{l}\|_{L^{\infty}(S_{\phi_h}(0, \theta))} \leq C_0\|w\|_{L^{\infty}(S_{\phi_h}(0, 2\theta))} (\theta^{\frac{1}{2}})^{1+\alpha_0}
&\leq& C_0 (\theta^{\frac{1}{2}})^{1+\alpha_0}\nonumber\\
&\leq& \frac{1}{2}(\theta^{\frac{1}{2}})^{1+\alpha},
\label{alpha0}
\end{eqnarray}
provided that
 $$C_0\theta_0^{\frac{\alpha_0-\alpha}{2}}\leq 1/2.$$
We will show that, by choosing $\theta\leq \theta_0$ where $\theta_0$ is small, we have
\begin{equation}\|w-v\|_{L^{\infty}(S_{\phi_h}(0, 2\theta))} \leq \frac{1}{2}(\theta^{\frac{1}{2}})^{1+\alpha}.
 \label{w-eq}
\end{equation}
Combining this with (\ref{alpha0}), we obtain
$$\|v-\bar{l}\|_{L^{\infty}(S_{\phi_h}(0, \theta))}\leq (\theta^{\frac{1}{2}})^{1+\alpha}.$$
Now, let
$$l_{m+1}(x):= l_m(x) + (h^{1/2})^{1+\alpha} \bar{l}(h^{-1/2}A_h x).$$
Then, for $x\in S_{\theta^{m+ 1}}= S_{\theta h}$, we have $h^{-1/2}A_h x\in S_{\phi_h}(0,\theta)$
and 
$$(u-l_{m+1})(x) = u(x)- l_m(x) - (h^{1/2})^{1+\alpha} \bar{l}(h^{-1/2}A_h x)= (h^{1/2})^{1+\alpha}(v- \bar{l})(h^{-1/2}A_h x).$$
Thus
$$\|u-l_{m+1}\|_{L^{\infty}(S_{\theta^{m+1}})} = (h^{1/2})^{1+\alpha}\|v- \bar{l}\|_{L^{\infty}(S_{\phi_h}(0,\theta))}\leq 
 (h^{1/2})^{1+\alpha}  (\theta^{1/2})^{1+\alpha}=  (\theta^{\frac{m+1}{2}})^{1+\alpha},$$
 proving (i). On the other hand, we have
 $$l_{m+1} (x) = b_{m+1}x_n$$
 where, by (\ref{Amap}) $$b_{m+1}:= b_m + (h^{1/2})^{1+\alpha} h^{-1/2} \bar{b} = b_m + h^{\alpha/2}\bar{b}.$$
 Therefore, the claim is established since (ii) follows from (\ref{b-bar}) and
 $$\abs{b_{m+1}-b_m}= h^{\alpha/2}\abs{\bar{b}}\leq C_{0}\theta^{m\alpha/2}.$$
 It remains to prove (\ref{w-eq}). We will use the ABP estimate to $w-v$ which solves
 \begin{equation*}
 \left\{
 \begin{alignedat}{2}
   \Phi_h^{ij}(w-v)_{ij} ~& = -f_h ~&&\text{in} ~ S_{\phi_h}(0, 2\theta), \\\
w-v &= \varphi_h-v~&&\text{on}~\p S_{\phi_h}(0, 2\theta).
 \end{alignedat} 
  \right.
\end{equation*} 
By this estimate and the way $\varphi_h$ is defined, we have
\begin{multline*}\|w-v\|_{L^{\infty}(S_{\phi_h}(0, 2\theta))} \leq \|v\|_{L^{\infty}(\p S_{\phi_h}(0, 2\theta)\cap \p\Omega_h)}
+ C(n) diam (S_{\phi_h}(0, 2\theta)) \|\frac{f_h}{(\det \Phi_h)^{\frac{1}{n}}}\|_{L^{n}(S_{\phi_h}(0, 2\theta))}\\=: (I) + (II).
\end{multline*}
To estimate (I), we denote $y= h^{1/2}A_{h}^{-1}x$ when $x\in \p S_{\phi_h}(0, 2\theta)\cap \p\Omega_h.$
Then $y\in \p S_{\phi}(0, 2\theta)\cap\p\Omega$ and moreover,
$$y_n = h^{1/2}x_n,~ y^{'}-\nu_h y_n = h^{1/2} x^{'}.$$
Noting that $x\in  \p S_{\phi_h}(0, 1)\cap \p\Omega_h\subset B_{k^{-1}}$, we have by (\ref{Amap})
$$\abs{y}\leq k^{-1}h^{1/2}\abs{log h}\abs{x} \leq h^{1/4}\leq \rho$$
if $h=\theta^m$ is small. This is clearly satisfied when $\theta_0$ is small. 

Since $\Omega$ has an interior tangent ball of radius $\rho$, we have
$$\abs{y_n}\leq \rho^{-1}|y^{'}|^2.$$
Therefore
$$\abs{v_h y_n}\leq k^{-1}\abs{log h} \rho^{-1} |y^{'}|^2 \leq k^{-1}\rho^{-1} h^{1/4}\abs{log h} |y^{'}|\leq \frac{1}{2}|y^{'}|$$
and consequently,
$$\frac{1}{2}|y^{'}|\leq |h^{1/2} x^{'}|\leq \frac{3}{2}|y^{'}|.$$
From (\ref{loc-h})
$$\tilde\rho |x^{'}|^2 \leq \phi_h(x)\leq 2\theta,$$ we have
$$|y^{'}|\leq 2 h^{1/2} |x^{'}|\leq 2(2\tilde\rho^{-1})^{1/2} (\theta h)^{1/2}.$$
By (ii) and $b_0=0$, we have
$$\abs{b_m}\leq \sum_{j=1}^{m}\abs{b_j-b_{j-1}}\leq \sum_{j=1}^{\infty} C_0 (\theta^{\theta/2})^{j-1}= \frac{C_0}{1-\theta^{\alpha/2}}\leq 2C_0$$
if $$\theta_0^{\alpha/2}\leq 1/2.$$
Now, we obtain from the definition of $v$ that
$$h^{\frac{1+\alpha}{2}} |v(x)| = |(u-l_m)(y)| \leq \abs{u(y)} + 2C_0\abs{y_n} \leq \delta |y^{'}|^{1+\gamma} + 2C_0 \rho^{-1} |y^{'}|^2
= |y^{'}|^{1+\gamma}(\delta + 2C_0 \rho^{-1} |y^{'}|^{1-\gamma}).$$
Using $|y^{'}|\leq C \theta^{1/2}$ and $\gamma\geq \alpha$, we find
$$v(x) \leq \frac{C ((\theta h)^{1/2})^{1+\gamma}(\delta +\theta^{\frac{1-\gamma}{2}})}{h^{\frac{1+\alpha}{2}}} = Ch^{\gamma-\alpha}\theta^{\frac{1+\gamma}{2}} (\theta^{\frac{1+\alpha}{2}} +\theta^{\frac{1-\gamma}{2}})\leq Ch^{\gamma-\alpha}\theta\leq
\frac{1}{4} (\theta^{1/2})^{1+\alpha}$$
if $\theta_0$ is small. We then obtain
$$(I) \leq \frac{1}{4} (\theta^{1/2})^{1+\alpha}.$$
To estimate (II), we recall $\delta = (\theta^{1/2})^{1+\alpha}$
and
$$S_{\phi_h}(0, 2\theta)\subset B_{C(2\theta)^{1/2} \abs{log 2\theta}};~ \abs{S_{\phi_h}(0, 2\theta)}\leq C(2\theta)^{n/2}.$$
Since $$\det \Phi_h = (\det D^2\phi_h)^{n-1}\geq \lambda^{n-1},$$ we therefore obtain from H\"older inequality that
\begin{eqnarray*}
 (II) &\leq& \frac{ C(n)}{\lambda^{\frac{n-1}{n}}} diam (S_{\phi_h}(0, 2\theta))
 \|f_h\|_{L^{n}(S_{\phi_h}(0, 2\theta))}\\
 &\leq& C(n, \lambda)diam (S_{\phi_h}(0, 2\theta)) \abs{S_{\phi_h}(0, 2\theta)}^{\frac{1}{n}-\frac{1}{p}}
 \| f_h\|_{L^{p}(S_{\phi_h}(0, 2\theta))} \\
 &\leq & C\delta \theta^{1/2}\abs{log 2\theta} (\theta^{1/2})^{1-n/p}= C(\theta^{1/2})^{1+\alpha}
 \abs{log 2\theta} (\theta^{1/2})^{2-n/p}\leq \frac{1}{4} (\theta^{1/2})^{1+\alpha} 
\end{eqnarray*}
if $\theta_0$ is small. It follows that
$$\|w-v\|_{L^{\infty}(S_{\phi_h}(0, 2\theta))} \leq (I) + (II)\leq \frac{1}{2}(\theta^{\frac{1}{2}})^{1+\alpha},$$
proving (\ref{w-eq}). The proof of our theorem is complete.
\end{proof}

\section{Global $C^{1,\alpha}$ estimates}
\label{proof-sec}
In this section, we will prove Theorem \ref{global-reg}.

\begin{proof}[Proof of Theorem \ref{global-reg}] We extend $\varphi$ to a $C^{1,\gamma}(\overline{\Omega})$ function in $\overline{\Omega}$.
By the ABP estimate, we have
 \begin{equation}
  \label{u-max}
  \norm{u}_{L^{\infty}(\Omega)} \leq C \left(\norm{f}_{L^{p}(\Omega)} + \|\varphi\|_{L^{\infty}(\overline{\Omega})}\right)
 \end{equation}
for some $C$ depending on $n, p, \rho, \lambda$. By multiplying $u$ by a suitable constant, we can assume that $$\norm{f}_{L^{p}(\Omega)}
+ \|\varphi\|_{C^{1,\gamma}(\overline{\Omega})}=1.$$

By using Guti\'errez-Nguyen's interior $C^{1,\alpha}$ estimates \cite{GN} and restricting our estimates in small balls of definite size around $\p\Omega$, we can assume throughout
 that $1-\e\leq g\leq 1+ \e$ where $\e$ is as in Theorem \ref{h-bdr-gradient}.
 
 Let $y\in \Omega $ with $r:=dist (y,\partial\Omega) \le c,$ for $c$ universal, and consider the maximal section $S_{\phi}(y, h)$ of $\phi$ centered at $y$, i.e.,
$$h=\sup\{t\,| \quad S_{\phi}(y,t)\subset \Omega\}.$$
Since $\phi$ is $C^{1, 1}$ on the boundary $\p\Omega$, by Caffarelli's strict convexity theorem, $\phi$ is strictly convex in $\Omega$. This implies the existence of the above maximal section $S_{\phi}(y, h)$ of $\phi$ centered at $y$ with $h>0$. 
By \cite[Proposition 3.2]{LS1} applied at the point $x_0\in \p S_{\phi}(y,h) \cap \p \Omega,$ we have
 \begin{equation} h^{1/2} \sim r,
\label{hr}
\end{equation}
and $ S_{\phi}(y,h)$ is equivalent to an ellipsoid $E$ i.e
$$cE \subset  S_{\phi}(y,h)-y \subset CE,$$
where
\begin{equation}E :=h^{1/2}A_{h}^{-1}B_1, \quad \mbox{with} \quad \|A_{h}\|, \|A_{ h}^{-1} \| \le C |\log  h|; \det A_{h}=1.
\label{eh}
\end{equation}
We denote $$\phi_y:=\phi-\phi(y)-\nabla \phi(y) (x-y).$$
The rescaling $\tilde \phi: \tilde S_1 \to \R$ of $u$ 
$$\tilde \phi(\tilde x):=\frac {1}{ h} \phi_y(T \tilde x) \quad \quad x=T\tilde x:=y+ h^{1/2}A_{h}^{-1}\tilde x,$$
satisfies
$$\det D^2\tilde \phi(\tilde x)=\tilde g(\tilde x):=g(T \tilde x),  $$
and
\begin{equation}
\label{normalsect}
B_c \subset \tilde S_1 \subset B_C, \quad \quad \tilde S_1=\bar h^{-1/2} A_{\bar h}(S_{y, \bar h}- y),
\end{equation}
where $\tilde S_1:= S_{\tilde \phi} (0, 1)$ represents the section of $\tilde \phi$ at the origin at height 1.

We define also the rescaling $\tilde u$ for $u$
$$\tilde u(\tilde x):= h^{-1/2}\left(u(T\tilde x)- u(x_{0})-\nabla u(x_0)(T\tilde x-x_0)\right),\quad \tilde x\in \tilde S_{1}.$$
Then $\tilde u$ solves
$$\tilde \Phi^{ij} \tilde u_{ij} = \tilde f(\tilde x):= h^{1/2} f(T\tilde x).$$
Now, we apply Guti\'errez-Nguyen's interior $C^{1,\alpha}$ estimates \cite{GN} to $\tilde u $ to obtain
$$\abs{D\tilde u (\tilde z_{1})-D\tilde u(\tilde z_{2})}\leq C\abs{\tilde z_{1}-\tilde z_{2}}^{\beta} \{\norm{\tilde u }_{L^{\infty}(\tilde S_{1})} + \norm{\tilde f}_{L^{p}(\tilde S_{1})}\},\quad\forall \tilde z_{1}, \tilde z_{2}\in \tilde S_{1/2},$$
for some small constant $\beta\in (0,1)$ depending only on $n, \lambda, \Lambda$.\\
By (\ref{normalsect}), we can decrease $\beta$ if necessary and thus we can assume that
$2\beta\leq \alpha$
where $\alpha\in (0,1)$ is the exponent in Theorem \ref{h-bdr-gradient}.
Note that, by (\ref{eh})
\begin{equation}
 \norm{\tilde f}_{L^{p}(\tilde S_{1})} = h^{1/2-\frac{n}{2p}}\norm{f}_{L^{p}(S_{y, \bar{h}})}.
 \label{scaled-lp}
 \end{equation}
We observe that (\ref{hr}) and (\ref{eh}) give
$$B_{C r\abs{log r}}(y)\supset  S_{\phi}(y,h) \supset  S_{\phi}(y,h/2)\supset B_{c\frac{r}{\abs{log r}}}(y)$$
and
$$diam ( S_{\phi}(y,h))\leq Cr\abs{log r}.$$
By Theorem \ref{h-bdr-gradient} applied to the original function $u$, (\ref{u-max}) and (\ref{hr}), we have
$$\norm{\tilde u }_{L^{\infty}(\tilde S_{1})} \leq C h^{-1/2} \left(\|u\|_{L^{\infty}(\Omega)} +\norm{f}_{L^{p}(\Omega)}
+ \|\varphi\|_{C^{1,\gamma}(\overline{\Omega})}\right)diam ( S_{\phi}(y,h))^{ 1+ \alpha} \leq
C  r^{\alpha}\abs{log r}^{1 +\alpha}.$$
Hence, using (\ref{scaled-lp}) and the fact that $\alpha\leq 1/2 (1-n/p)$, we get
$$\abs{D\tilde u (\tilde z_{1})-D\tilde u(\tilde z_{2})}\leq C\abs{\tilde z_{1}-\tilde z_{2}}^{\beta}
r^{\alpha}\abs{log r}^{1 +\alpha}~\forall \tilde z_{1}, \tilde z_{2}\in \tilde S_{1/2}.$$
 Rescaling back and using
$$\tilde z_1-\tilde z_2= h^{-1/2}A_{ h}(z_1-z_2),\quad h^{1/2}\sim r,$$
and the fact that
$$\abs{\tilde z_1-\tilde z_2}\leq \norm{ h^{-1/2}A_{ h}}\abs{z_1-z_2} \leq C h^{-1/2}\abs{\log h}\abs{z_1-z_2}\leq
C r^{-1}\abs{log r}\abs{z_1-z_2},$$
we find
\begin{eqnarray}|Du(z_1)-Du( z_2)| &=&|A_{h}(D\tilde u (\tilde z_{1})-D\tilde u(\tilde z_{2})|  \leq C\abs{\log h}(r^{-1}\abs{log r}\abs{z_1-z_2})^{\beta} 
r^{\alpha}\abs{log r}^{1 +\alpha} \nonumber\\&
\le&  |z_1-z_2|^{\beta} \quad \forall  z_1, z_2 \in   S_{\phi}(y,h/2).
\label{oscv}
\end{eqnarray}
Notice that this inequality holds also in the Euclidean ball $B_{c\frac{r}{\abs{log r}}}(y)\subset  S_{\phi}(y,h/2)$. Combining this with 
Theorem \ref{h-bdr-gradient}, we easily 
obtain  $$[Du]_{C^\beta(\bar \Omega)} \le C$$
and the desired global $C^{1,\beta}$ bounds for $u$.
\end{proof}
{\bf Acknowledgments.} The authors would like to thank the referee for constructive comments on the manuscript. The first author was partially supported by the Vietnam Institute for Advanced Study in Mathematics (VIASM), Hanoi, Vietnam. 

\end{document}